\documentclass[11pt,reqno]{amsart}
\usepackage{amsmath,amssymb}
\makeatletter \oddsidemargin.9375in \evensidemargin \oddsidemargin
\begin{document}
\begin{center}
\thispagestyle{empty} \setcounter{page}{1} {\large\bf The fixed
point alternative and the stability of ternary derivations
\vskip.20in {\Small \bf A. Ebadian and Sh.
Najafzadeh} \\[2mm]

{\footnotesize  Department of Mathematics,
Urmia University, Urmia, Iran\\

{\footnotesize  Department of Mathematics, University of Maragheh,
Maragheh, Iran\\}

 e-mail: {\tt  ebadian.ali@gmail.com }}}
\end{center}
\vskip 5mm \noindent{\footnotesize{\bf Abstract.} Using fixed point
methods, we prove the generalized Hyers--Ulam--Rassias stability of
ternary derivations in ternary Banach algebras for the generalized
Jensen--type functional equation
$$\mu f(\frac{x+y+z}{3})+\mu f(\frac{x-2y+z}{3})+\mu f(\frac{x+y-2z}{3})= f(\mu x)~.$$

\vskip.10in 

\vskip.10in
\newtheorem{df}{Definition}[section]
\newtheorem{rk}[df]{Remark}
\newtheorem{lem}[df]{Lemma}
\newtheorem{thm}[df]{Theorem}
\newtheorem{pro}[df]{Proposition}
\newtheorem{cor}[df]{Corollary}
\newtheorem{ex}[df]{Example}
\setcounter{section}{0} \numberwithin{equation}{section} \vskip
.2in
\begin{center}
\section{Introduction}
\end{center}
Ternary algebraic operations were considered in the 19 th century
by several mathematicians such as A. Cayley [1] who introduced the
notion of "cubic matrix" which in turn was generalized by
Kapranov, Gelfand and Zelevinskii in 1990 ([2]).\\
The comments on physical applications of ternary structures can be
found in [3--12].\\
Let $A$ be a Banach ternary algebra and $X$ be a Banach space. Then
$X$ is called a ternary Banach $A$-module, if module operations $A
\times A \times X \to X,$ $A \times X \times A \to X,$ and $X \times
A \times A \to X$ which are $\Bbb C$-linear in every variable.
Moreover  satisfy
$$[[xab]_X~cd]_X=[x[abc]_A~d]_X=[xa[bcd]_A]_X,$$
$$[[axb]_X~cd]_X=[a[xbc]_X~d]_X=[ax[bcd]_A]_X,$$
$$[[abx]_X~cd]_X=[a[bxc]_X~d]_X=[ab[xcd]_X]_X,$$
$$[abc]_A~xd]_X=[a[bcx]_X~d]_X=[ab[cxd]_X]_X,$$
$$[[abc]_A~dx]_X=[a[bcd]_A~x]_X=[ab[cdx]_X]_X$$
for all $x \in X$ and all $a,b,c,d \in A,$
$$\max\{\|xab\|,\|axb\|,\|abx\|\}\leq \|a\| \|b\| \|x\|$$
for all $x \in X$ and all $a,b \in A.$\\ Let $(A,[~]_A)$ be a Banach
ternary algebra over a scalar field $\Bbb R$ or $\Bbb C$ and
$(X,[~]_X)$ be a ternary Banach $A$-module. A linear mapping
$D:(A,[~]_A) \to (X,[~]_X)$ is called a ternary derivation, if
$$D([xyz]_A)=[D(x)yz]_X+[xD(y)z]_X+[xyD(z)]_X$$ for all
$x,y,z\in A.$\\
A linear mapping $D:(A,[~]_A) \to (X,[~]_X)$ is called a ternary
Jordan derivation, if
$$D([xxx]_A)=[D(x)xx]_X+[xD(x)x]_X+[xxD(x)]_X$$ for
all $x\in A.$\\
The stability of functional equations was first introduced  by S. M.
Ulam [13] in 1940. More precisely, he proposed the following
problem: Given a group $G_1,$ a metric group $(G_2,d)$ and a
positive number $\epsilon$, does there exist a $\delta>0$ such that
if a function $f:G_1\longrightarrow G_2$ satisfies the inequality
$d(f(xy),f(x)f(y))<\delta$ for all $x,y\in G1,$ then there exists a
homomorphism $T:G_1\to G_2$ such that $d(f(x), T(x))<\epsilon$ for
all $x\in G_1?$ As mentioned above, when this problem has a
solution, we say that the homomorphisms from $G_1$ to $G_2$ are
stable. In 1941, D. H. Hyers [14] gave a partial solution of
$Ulam^{,}s$ problem for the case of approximate additive mappings
under the assumption that $G_1$ and $G_2$ are Banach spaces. In
1950, T. Aoki [25] was the second author to treat this problem for
additive mappings (see also [16]). In 1978, Th. M. Rassias [17]
generalized the theorem of Hyers by considering the stability
problem with unbounded Cauchy differences. This phenomenon of
stability that was introduced by Th. M. Rassias [17] is called the
Hyers--Ulam--Rassias stability. According to Th. M. Rassias theorem:
\begin{thm}\label{t1} Let $f:{E}\longrightarrow{E'}$ be a mapping from
 a norm vector space ${E}$
into a Banach space ${E'}$ subject to the inequality
$$\|f(x+y)-f(x)-f(y)\|\leq \epsilon (\|x\|^p+\|y\|^p)$$
for all $x,y\in E,$ where $\epsilon$ and p are constants with
$\epsilon>0$ and $p<1.$ Then there exists a unique additive
mapping $T:{E}\longrightarrow{E'}$ such that
$$\|f(x)-T(x)\|\leq \frac{2\epsilon}{2-2^p}\|x\|^p$$ for all $x\in E.$
If $p<0$ then inequality $(1.3)$ holds for all $x,y\neq 0$, and
$(1.4)$ for $x\neq 0.$ Also, if the function $t\mapsto f(tx)$ from
$\Bbb R$ into $E'$ is continuous for each fixed $x\in E,$ then T
is linear.
\end{thm}
During the last decades several stability problems of functional
equations have been investigated by many mathematicians. A large
list of references concerning the stability of functional
equations can be found in  [18--29].\\
 C. Park [30]  has contributed works to the stability
problem of ternary homomorphisms and
ternary derivations (see also [31]).\\

Recently, C$\breve{a}$dariu and Radu  applied the fixed point method
to the investigation of the  functional equations. (see also
[32--38]).

 In this paper, we will adopt the
fixed point alternative of C\u{a}dariu  and Radu to prove the
generalized Hyers--Ulam--Rassias stability of ternary derivations on
ternary Banach algebras associated with the following functional
equation
$$\mu f(\frac{x+y+z}{3})+\mu f(\frac{x-2y+z}{3})+\mu f(\frac{x+y-2z}{3})= f(\mu x)~.$$

Throughout this paper, assume that $(A,[~]_A)$ is a
 ternary Banach algebra and $X$ is a  ternary Banach $A-$module.
 \vskip 5mm
\section{Main Results}

Before proceeding to the main results, we will state the following
theorem.

\begin{thm}\label{t2}(the alternative of fixed point [32]).
Suppose that we are given a complete generalized metric space
$(\Omega,d)$ and a strictly contractive mapping
$T:\Omega\rightarrow\Omega$ with Lipschitz constant $L$. Then for
each given $x\in\Omega$, either

$d(T^m x, T^{m+1} x)=\infty~$ for all $m\geq0,$\\
 or other exists a natural number $m_{0}$ such that\\
$\star\hspace{.25cm} d(T^m x, T^{m+1} x)<\infty ~$for all $m \geq m_{0};$ \\
 $\star\hspace{.1cm}$ the sequence $\{T^m x\}$ is convergent to a fixed point $y^*$ of $~T$;\\
$\star\hspace{.2cm} y^*$is the unique fixed point of $~T$ in the
set $~\Lambda=\{y\in\Omega:d(T^{m_{0}} x, y)<\infty\};$\\
$\star\hspace{.2cm} d(y,y^*)\leq\frac{1}{1-L}d(y, Ty)$ for all
$~y\in\Lambda.$
\end{thm}

We start our work with the following theorem which establishes the
generalized Hyers--Ulam--Rassias stability of ternary derivations.

\begin{thm}\label{t1}
Let $f:A\to X$ be a mapping for which there exists a function
$\phi:A^6\to [0,\infty)$ such that

$$\|\mu f(\frac{x+y+z}{3})+\mu f(\frac{x-2y+z}{3})+\mu f(\frac{x+y-2z}{3})- f(\mu
x)$$
$$+f([abc]_A)-[f(a)bc]_X-[af(b)c]_X-[abf(c)]_X\|_X\leq\phi(x,y,z,a,b,c),\eqno(2.1)$$
for all $\mu \in \Bbb T$ and all $x,y,z,a,b,c \in A.$ If  there
exists an $L<1$ such that $$\phi(x,y,z,a,b,c)\leq 3L
\phi(\frac{x}{3},\frac{y}{3},\frac{z}{3},\frac{a}{3},\frac{a}{3},\frac{c}{3})$$
for all $x,y,z,a,b,c\in A,$ then there exists a unique ternary
derivation $D:A\to X$ such that
$$\|f(x)-D(x)\|_B \leq \frac{L}{1-L}\phi(x,0,0,0,0,0)\eqno(2.2)$$
for all $x \in A.$
\end{thm}
\begin{proof}It follows from
$$\phi(x,y,z,a,b,c)\leq 3L
\phi(\frac{x}{3},\frac{y}{3},\frac{z}{3},\frac{a}{3},\frac{b}{3},\frac{c}{3})$$
that
$$lim_j 3^{-j}\phi(3^jx,3^jy,3^iz,3^ia,3^ib,3^ic)=0 \eqno(2.3)$$ for all $x,y,z,a,b,c\in
A.$\\
Put $\mu =1, y=z=a=b=c=0$ in (2.1) to obtain
$$\|3f(\frac{x}{3})-f(x)\|_B\leq \phi(x,0,0,0,0,0)\eqno (2.4)$$
for all $x\in A.$ Hence,
$$\|\frac{1}{3}f(3x)-f(x)\|_B\leq \frac{1}{3} \phi(3x,0,0,0,0,0)\leq L\phi(x,0,0,0,0,0)\eqno (2.5)$$
for all $x\in A.$\\
Consider the set $X':=\{g\mid g:A\to B\}$ and introduce the
generalized metric on $X'$:
$$d(h,g):=inf\{C\in \Bbb R^+:\|g(x)-h(x)\|_B\leq C\phi(x,0,0,0,0,0) \forall x\in A\}.$$
It is easy to show that $(X',d)$ is complete. Now we define  the
linear mapping $J:X'\to X'$ by $$J(h)(x)=\frac{1}{3}h(3x)$$ for all
$x\in A$. By Theorem 3.1 of [32], $$d(J(g),J(h))\leq Ld(g,h)$$ for
all $g,h\in X'.$\\
It follows from (2.5) that  $$d(f,J(f))\leq L.$$ By Theorem 1.2, $J$
has a unique fixed point in the set $X_1:=\{h\in X': d(f,h)<
\infty\}$. Let $D$ be the fixed point of $J$. $D$ is the unique
mapping with
$$D(3x)=3D(x)$$ for all $x\in A$ satisfying there exists $C\in
(0,\infty)$ such that
$$\|D(x)-f(x)\|_B\leq C\phi(x,0,0,0,0,0)$$ for all $x\in A$. On the other hand we
have $lim_n d(J^n(f),D)=0$. It follows that
$$lim_n\frac{1}{3^n}{f(3^nx)}=D(x)\eqno (2.6)$$
for all $x\in A$. It follows from $d(f,D)\leq
\frac{1}{1-L}d(f,J(f)),$ that $$d(f,D)\leq \frac{L}{1-L}.$$ This
implies the inequality (2.2). It follows from (2.1), (2.3) and (2.6)
that
\begin{align*}\|&D(\frac{x+y+z}{3})+ D(\frac{x-2y+z}{3})+
D(\frac{x+y-2z}{3})- D( x)\|_X \\
&=lim_n
\frac{1}{3^n}\|f(3^{n-1}(x+y+z))+f(3^{n-1}(x-2y+z))+f(3^{n-1}(x+y-2z))-f(3^nx)\|_X\\
&\leq lim_n\frac{1}{3^n}\phi(3^nx,3^ny,3^nz,3^na,3^nb,3^nc)=0
\end{align*}
for all $x,y,z \in A.$ So $$D(\frac{x+y+z}{3})+ D(\frac{x-2y+z}{3})+
D(\frac{x+y-2z}{3})= D( x)$$ for all $x,y,z \in A.$ Put
$w=\frac{x+y+z}{3}, t=\frac{x-2y+z}{3}$ and $s=\frac{x+y-2z}{3}$ in
above equation, we get $D(w+t+s)=D(w)+D(t)+D(s)$ for all $w,t,s\in
A.$ Hence, $D$ is Cauchy additive.  By putting $y=z=x, a=b=c=0$ in
(2.1), we have
$$\|\mu f(x)-f(\mu x)\|_X\leq \phi(x,x,x,0,0,0)$$
for all $x\in A$. It follows that
$$\|D(\mu x)-\mu D(x)\|_X=lim_n
\frac{1}{3^n}\|f(\mu 3^n x)-\mu f(3^n x)\|_X\leq
lim_n\frac{1}{3^n}\phi(3^nx,3^nx,3^nx,3^na,3^nb,3^nc)=0$$ for all
$\mu \in \Bbb T$, and all $x\in A.$ One can show that the mapping
$D:A\to B$ is $\Bbb C-$linear.  It follows from (2.1) that
\begin{align*}\|&D([xyz]_A)-[D(x)yz]_X-[xD(y)z]_X-[xyD(z)]_X\|_X\\
&=lim_n\|\frac{1}{27^n}D([3^nx3^ny3^nz]_A)-\frac{1}{27^n}([D(3^nx)3^ny3^nz]_X\\
&+[3^nxD(3^ny)3^nz]_X+[3^nx3^nyD(3^nz)]_X)\|_X\leq lim_n \frac{1}{27^n}\phi(0,0,0,3^nx,3^ny,3^nz)\\
&\leq lim_n \frac{1}{3^n}\phi(0,0,0,3^nx,3^ny,3^nz)\\
&=0
\end{align*}
for all $x,y,z \in A.$ So
$$D([xyz]_A)=[D(x)yz]_X+[xD(y)z]_X+[xyD(z)]_X$$ for all $x,y,z \in
A.$ Hence,  $D:A\to X$ is a ternary derivation satisfying (2.2), as
desired.

\end{proof}
We prove the following Hyers--Ulam--Rassias stability problem for
ternary derivations on ternary Banach algebras.

\begin{cor}\label{t2}
Let $p\in (0,1), \theta \in [0,\infty)$ be real numbers. Suppose
$f:A \to X$ satisfies
$$ \|\mu f(\frac{x+y+z}{3})+\mu f(\frac{x-2y+z}{3})+\mu f(\frac{x+y-2z}{3})- f(\mu x)\|_X \leq \theta(\|x\|^p_A+\|y\|^p_A+\|z\|^p_A),$$
$$\|f([abc]_A)-[f(a)bc]_X-[af(b)c]_X-[abf(c)]_X\|_X, \leq \theta(\|a\|^p_A+\|b\|^p_A+\|c\|^p_A),$$
for all $\mu \in \Bbb T$ and all $a,b,c,x,y,z \in A.$ Then there
exists a unique ternary derivation $D:A\to X$ such that
$$\|f(x)-D(x)\|_B \leq \frac{2^p\theta}{2-2^p}\|x\|^p_A$$
for all $x \in A.$
\end{cor}
\begin{proof}
Setting
$\phi(x,y,z,a,b,c):=\theta(\|x\|^p_A+\|y\|^p_A+\|z\|^p_A+\|a\|^p_A+\|b\|^p_A+\|c\|^p_A)$
all $x,y,z,a,b,c \in A.$ Then by $L=2^{p-1}$, we get the desired
result.
\end{proof}

\begin{thm}\label{t3}
Let $f:A\to X$ be a mapping for which there exists a function
$\phi:A^6\to [0,\infty)$ satisfying (2.1).  If there exists an $L<1$
such that $\phi(x,y,z,a,b,c)\leq \frac{1}{3}L
\phi(3x,3y,3z,3a,3b,3c)$ for all $x,y,z,a,b,c\in A,$ then there
exists a unique ternary derivation $D:A\to X$ such that
$$\|f(x)-D(x)\|_X \leq \frac{L}{3-3L}\phi(x,0,0,0,0,0)\eqno(2.7)$$
for all $x \in A.$
\end{thm}
\begin{proof}It follows from (2.4) that
$$\|3f(\frac{x}{3})-f(x)\|_X\leq \phi(\frac{x}{3},0,0,0,0,0)\leq \frac{L}{3}\phi(x,0,0,0,0,0)\eqno (2.8)$$
for all $x\in A.$ We consider  the linear mapping $J:X'\to X'$ such
that
$$J(h)(x)=3h(\frac{x}{3})$$ for all $x\in A$.
It follows from (2.9) that  $$d(f,J(f))\leq \frac {L}{3}.$$ By
Theorem 2.1, $J$ has a unique fixed point in the set $X_1:=\{h\in
X': d(f,h)< \infty\}$. Let $D$ be the fixed point of $J$, that is,
$$D(3x)=3D(x)$$ for all $x\in A$ satisfying there exists $C\in
(0,\infty)$ such that
$$\|D(x)-f(x)\|_X\leq C\phi(x,0,0,0,0,0)$$ for all $x\in A$. We have  $ d(J^n(f),D)\to 0$ as $n\to 0$. This implies
the equality
$$lim_n3^nf(\frac{x}{3^n})=D(x)\eqno (2.9)$$
for all $x\in A$. It follows from $d(f,D)\leq
\frac{1}{1-L}d(f,J(f)),$ that $$d(f,D)\leq \frac{L}{3-3L},$$ which
 implies the inequality (2.7). The rest of the proof is similar to the proof of Theorem 2.2.
\end{proof}

\begin{cor}\label{t4}
Let $p\in (3,\infty), \theta \in [0,\infty)$ be real numbers.
Suppose $f:A \to X$ satisfies

$$ \|\mu f(\frac{x+y+z}{3})+\mu f(\frac{x-2y+z}{3})+\mu f(\frac{x+y-2z}{3})- f(\mu x)\|_X \leq \theta(\|x\|^p_A+\|y\|^p_A+\|z\|^p_A),$$
$$\|f([abc]_A)-[f(a)bc]_X-[af(b)c]_X-[abf(c)]_X\|_X, \leq \theta(\|a\|^p_A+\|b\|^p_A+\|c\|^p_A),$$
for all $\mu \in \Bbb T$ and all $a,b,c,x,y,z \in A.$

Then there exists a unique ternary derivation $D:A\to X$ such that
$$\|f(x)-D(x)\|_X \leq \frac{\theta}{3^p-3}\|x\|^p_A$$
for all $x \in A.$
\end{cor}
\begin{proof}
Setting
$\phi(x,y,z,a,b,c):=\theta(\|x\|^p_A+\|y\|^p_A+\|z\|^p_A+\|a\|^p_A+\|b\|^p_A+\|c\|^p_A)$
for all $x,y,z,a,b,c \in A.$ Then by $L=3^{1-p}$, we get the desired
result.
\end{proof}

Now we investigate the generalized Hyers--Ulam--Rassias stability of
 Jordan ternary derivations.

\begin{thm}\label{t1}
Let $f:A\to X$ be a mapping for which there exists a function
$\phi:A^4\to [0,\infty)$ such that

$$\|\mu f(\frac{x+y+z}{3})+\mu f(\frac{x-2y+z}{3})+\mu f(\frac{x+y-2z}{3})- f(\mu
x)$$
$$+f([aaa]_A)-[f(a)aa]_X-[af(a)a]_X-[aaf(a)]_X\|_X\leq\phi(x,y,z,a),\eqno(2.10)$$
for all $\mu \in \Bbb T$ and all $x,y,z,a,b,c \in A.$ If  there
exists an $L<1$ such that $$\phi(x,y,z,a)\leq 3L
\phi(\frac{x}{3},\frac{y}{3},\frac{z}{3},\frac{a}{3})$$ for all
$x,y,z,a\in A,$ then there exists a unique Jordan ternary derivation
$D:A\to X$ such that
$$\|f(x)-D(x)\|_X \leq \frac{L}{1-L}\phi(x,0,0,0)\eqno(2.11)$$
for all $x \in A.$
\end{thm}

\begin{proof}
By the same reasoning as the proof of Theorem 2.2, there exists a
unique involutive $\Bbb C-$linear mapping $D:A \to X$ satisfying
(2.11). The mapping $D$ is given by
$$D(x)=lim_n\frac{1}{3^n}{f(3^nx)}$$
for all $x\in A.$ The relation (2.10) follows that

\begin{align*}\|&D([xxx]_A)-[D(x)xx]_X-[xD(x)x]_X-[xxD(x)]_X\|_X\\
&=lim_n\|\frac{1}{27^n}D([3^nx3^nx3^nx]_A)-\frac{1}{27^n}([D(3^nx)3^nx3^nx]_X\\
&+[3^nxD(3^nx)3^nx]_X+[3^nx3^nxD(3^nx)]_X)\|_X\leq lim_n \frac{1}{27^n}\phi(0,0,0,3^nx)\\
&\leq lim_n \frac{1}{3^n}\phi(0,0,0,3^nx)\\
&=0
\end{align*}
for all $x \in A.$ So
$$D([xxx]_A)=[D(x)xx]_X+[xD(x)x]_X+[xxD(x)]_X$$ for all $x \in
A.$ Hence,  $D:A\to X$ is a Jordan ternary derivation satisfying
(2.11), as desired.
\end{proof}

We prove the following Hyers--Ulam--Rassias stability problem for
Jordan  ternary derivations on ternary Banach algebras.

\begin{cor}\label{t6}
Let $p\in (0,1), \theta \in [0,\infty)$ be real numbers. Suppose
$f:A \to B$ satisfies
$$\|\mu f(\frac{x+y+z}{3})+\mu f(\frac{x-2y+z}{3})+\mu f(\frac{x+y-2z}{3})- f(\mu x)\|_B \leq \theta(\|x\|^p_A+\|y\|^p_A+\|z\|^p_A),$$
$$\|f([xxx]_A)-[f(x)xx]_X-[xf(x)x]_X-[xxf(x)]_X\|_X\leq 3\theta(\|x\|^p_A)$$
for all $\mu \in \Bbb T$, and all $x,y,z \in A.$ Then there exists a
unique Jordan ternary derivation $D:A\to X$ such that
$$\|f(x)-D(x)\|_X \leq \frac{2^p\theta}{2-2^p}\|x\|^p_A$$
for all $x \in A.$
\end{cor}
\begin{proof}
Setting
$\phi(x,y,z,a):=\theta(\|x\|^p_A+\|y\|^p_A+\|z\|^p_A+\|a\|^p_A)$ all
$x,y,z,a \in A.$ Then by $L=2^{p-1}$, we get the desired result.
\end{proof}

\begin{thm}\label{t3}
Let $f:A\to X$ be a mapping for which there exists a function
$\phi:A^4\to [0,\infty)$ satisfying (2.10).  If there exists an
$L<1$ such that $\phi(x,y,z,a)\leq \frac{1}{3}L \phi(3x,3y,3z,3a)$
for all $x,y,z,a\in A,$ then there exists a unique Jordan ternary
derivation $D:A\to X$ such that
$$\|f(x)-D(x)\|_X \leq \frac{L}{3-3L}\phi(x,0,0,0)$$
for all $x \in A.$
\end{thm}

\begin{proof}The proof is similar to the proofs of Theorems 2.4 and 2.6.
\end{proof}

\begin{cor}\label{t8}
Let $p\in (3,\infty), \theta \in [0,\infty)$ be real numbers.
Suppose $f:A \to X$ satisfies
$$\|\mu f(\frac{x+y+z}{3})+\mu f(\frac{x-2y+z}{3})+\mu f(\frac{x+y-2z}{3})- f(\mu x)\|_B \leq \theta(\|x\|^p_A+\|y\|^p_A+\|z\|^p_A),$$
$$\|f([xxx]_A)-[f(x)xx]_X-[xf(x)x]_X-[xxf(x)]_X\|_X\leq 3\theta(\|x\|^p_A)$$
 for all $\mu \in
\Bbb T$, and  all $x, y,z \in A.$   Then there exists a unique
Jordan ternary derivation $D:A\to X$ such that
$$\|f(x)-D(x)\|_X \leq \frac{\theta}{3^p-3}\|x\|^p_A$$
for all $x \in A.$
\end{cor}
\begin{proof}
Setting
$\phi(x,y,z,a):=\theta(\|x\|^p_A+\|y\|^p_A+\|z\|^p_A+\|a\|^p_A)$ all
$x,y,z \in A$ in above theorem. Then by $L=3^{1-p}$, we get the
desired result.
\end{proof}

{\small

}

\begin{thebibliography}{99}
\bibitem{Ca} A. Cayley, On the $34$ concomitants of the ternary cubic, {\it Am. J. Math.} 4, 1 (1881).

\bibitem{Ka} M. Kapranov, I. M. Gelfand and A. Zelevinskii, Discrimininants, {\it Resultants and Multidimensional Determinants, Birkhauser, Berlin,} 1994.

\bibitem{Ab} V. Abramov, R. Kerner and B. Le Roy, Hypersymmetry a $Z_3$ graded generalization of supersymmetry, {\it J. Math. Phys.} 38, 1650 (1997).


\bibitem{Ba} F. Bagarello and G. Morchio, Dynamics of mean-field spin models from
basic results in abstract differential equations, {\it J. Stat.
Phys.} 66 (1992) 849-866. MR1151983 (93c:82034)

\bibitem{Ba} N. Bazunova, A. Borowiec and R. Kerner, Universal
differential calculus on ternary algebras, {\it Lett. Matt. Phys.}
67 (2004), no. 3, 195-206.

\bibitem{Ha} R. Haag and D. Kastler, An algebraic approach to quantum field
theory, {\it J. Math. Phys.} 5 (1964) 848-861. MR0165864

\bibitem{Ke} R. Kerner, Ternary algebraic structures and their applications in physics, {\it Univ. P. M. Curie preprint,
Paris} (2000), http://arxiv.org/list/math-ph/0011.


\bibitem{Ke} R. Kerner, The cubic chessboard: {\it Geometry and physics, Class.
Quantum Grav.} 14, A203 (1997).

\bibitem{Se} G. L. Sewell, Quantum Mechanics and its Emergent Macrophysics,
{\it Princeton Univ. Press, Princeton, NJ,} 2002. MR1919619
(2004b:82001)
\bibitem{R} L. Takhtajan, On foundation of the generalized Nambu mechanics,
{\it Comm. Math. Phys.} 160 (1994), no. 2, 295-315.

\bibitem{V} L. Vainerman and R. Kerner, On special classes of n-algebras, {\it J.
Math. Phys.} 37 (1996), no. 5, 2553–-2565.

\bibitem{Ze} H. Zettl, A characterization of ternary rings of operators, {\it Adv.
Math.} 48 (1983) 117-143. MR0700979 (84h:46093)

\bibitem{U} S. M. Ulam, Problems in Modern Mathematics,{\it
Chapter VI, science ed. Wiley, New York,} 1940.

\bibitem{H} D. H. Hyers, On the stability of the linear functional
equation, {\it Proc. Natl. Acad. Sci.} 27 (1941) 222-224.

\bibitem{Ao} T. Aoki, On the stability of the linear transformation in Banach spaces,
{\it J. Math. Soc. Japan.} 2(1950), 64-66.

\bibitem{BOU} D. G. Bourgin, Approximately isometric and multiplicative
transformations on continuous function rings,{\it Duke Math. J.}
16, (1949). 385--397.

\bibitem{R} Th. M. Rassias, On the stability of the linear mapping in Banach
spaces, {\it Proc. Amer. Math. Soc.} 72 (1978) 297-300.


\bibitem{E} M. Bavand Savadkouhi, M. Eshaghi Gordji,
N. Ghobadipour and J. M. Rassias, Approximate ternary Jordan
derivations on Banach ternary algebras,  Journal of Mathematical
Phisics, 50, 042303 (2009).

\bibitem{Ra} J. M. Rassias, Solution of a problem of Ulam, {\it J. Approx. Theory} 57 (1989), no. 3, 268-273.


\bibitem{P-R} J. M.  Rassias and  H. M. Kim,  Approximate homomorphisms and
derivations between $C\sp *$-ternary algebras.{\it J. Math. Phys.}
49 (2008), no. 6, 063507, 10 pp. 46Lxx (39B82)

\bibitem{Ch} P. W. Cholewa, Remarks on the stability of functional
equations, {\it Aequationes Math.} 27 (1984) 76-86. MR0758860
(86d:39016)

\bibitem{Cz} S. Czerwik, Stability of functional equations of
Ulam-Hyers-Rassias type, {\it Hadronic Press,} 2003.

\bibitem{Ga} Z. Gajda, On stability of additive mappings, {\it Internat. J. Math.
Math. Sci.} 14(1991) 431-434.

\bibitem{H} D. H. Hyers, G. Isac and Th. M. Rassias, Stability of functional
Equations in Several Variables, {\it Birkhauser, Boston, Basel,
Berlin,} 1998.


\bibitem {R6} Th. M. Rassias, On the stability of functional
equations and a problem of Ulam,{\it Acta Math. Appl.} 62 (2000)
23-130. MR1778016 (2001j:39042)


\bibitem {R5} Th. M. Rassias, On the stability of functional
equations in Banach spaces,{\it J. Math. Anal. Appl.} 251 (2000)
264-284. MR1790409 (2003b:39036)


\bibitem {R7}  Th. M. Rassias, The problem of S.M.Ulam for
approximately multiplicative mappings,{\it J. Math. Anal. Appl.}
246(2)(2000),352-378.





\bibitem{I-Ra} G. Isac and Th. M. Rassias, On the Hyers-Ulam stability of
Ã-additive mappings, {\it J. Approx. Theorey} 72 (1993), 131-137.

\bibitem {Cz} M. Eshaghi Gordji, H. Khodaei,  Solution and stability of generalized mixed type
cubic, quadratic and additive functional equation in quasi--Banach
spaces, {\it Nonlinear Analysis.} (2009), article in press.

\bibitem{P} C. Park, Isomorphisms between C*-ternary algebras, {\it J. Math. Anal.
Appl.} 327 (2007), 101-115.


\bibitem{P} C. Park, Homomorphisms between Poisson JC*-algebras, {\it Bull. Braz. Math. Soc.} 36 (2005) 79-97.
MR2132832 (2005m:39047)

\bibitem{ChuS} L. C\u adariu and V. Radu, On the stability of the Cauchy functional equation:
a fixed point approach, {\it Grazer Mathematische Berichte} {\bf
346} (2004),  43--52.

\bibitem {Ra} V. Radu, The fixed point alternative and the stability of functional equations,
{\it Fixed Point Theory} {\bf 4} (2003),  91--96.


\bibitem {R3} I.A. Rus, {\it Principles and Applications of Fixed Point Theory,} Ed.
Dacia, Cluj-Napoca, 1979 (in Romanian).



\bibitem{Ac} L. C$\breve{a}$dariu, V. Radu, {\it The fixed points method for the stability of some functional
equations,} Carpathian Journal of Mathematics {\bf 23} (2007),
 63--72.




\bibitem{Ac} L. C$\breve{a}$dariu, V. Radu, {\it  Fixed points and the stability of quadratic
functional equations,} Analele Universitatii de Vest din Timisoara
{\bf 41} (2003),  25--48.


\bibitem{Chu-S} L. C$\breve{a}$dariu and V. Radu, Fixed points and the stability of Jensen's functional equation,
{\it J. Inequal. Pure Appl. Math.} {\bf 4 } (2003), Art. ID 4.



\bibitem{Ao} S. Rolewicz, Metric Linear Spaces,{\it PWN-Polish Sci. Publ./Reidel,
Warszawa/Dordrecht,} 1984.

\end{thebibliography}
\end{document}